\documentclass[10pt,a4paper]{article}
\usepackage[utf8]{inputenc}
\usepackage[T1]{fontenc}
\usepackage{enumitem}
\usepackage[english]{babel}
\usepackage{amsmath}
\usepackage{amssymb}
\usepackage{mathrsfs}
\usepackage{amsthm}
\usepackage{amsfonts}
\usepackage{stmaryrd}
\usepackage{geometry}
\usepackage{hyperref}
\hypersetup{
    colorlinks=true,
    linkcolor=blue,
    citecolor=blue}

\newtheorem{thm}{Theorem}[section]
\newtheorem{lem}[thm]{Lemma}
\newtheorem{prop}[thm]{Proposition}
\newtheorem{rem}[thm]{Remark}
\newtheorem{rems}[thm]{Remarks}
\newtheorem{defi}[thm]{Definition}
\newtheorem*{defin}{Definition}
\newtheorem{cor}[thm]{Corollary}

\title{On some stable representations of hyperbolic groups}
\author{Ulysse Remfort-Aurat}
\date{}

\newcommand{\R}{\mathbb{R}}
\newcommand{\N}{\mathbb{N}}

\newcommand{\Or}{\tau_{\rho,x}}
\newcommand{\B}{\mathcal{B}}
\newcommand{\p}{\mathcal{P}}
\newcommand{\w}{\mathcal{W}}

\begin{document}

\maketitle

\begin{abstract}
Let $\Gamma$ be a hyperbolic group and $G$ be the isometry group of a Gromov-hyperbolic, proper and geodesic metric space. We study the action of the outer automorphism group $Out(\Gamma)$ on the set $X(\Gamma,G)$ of conjugacy classes of representations of $\Gamma$ into $G$. We construct a family of  $Out(\Gamma)$-invariant subsets of $X(\Gamma,G)$ which contains (stricly or not) the set of conjugacy classes of quasi-convex representations and give a sufficient condition for the induced action to be properly discontinuous. Finally, we give a criterion for a representation to have discrete image and finite kernel and use it when $G=Isom^+(\mathbb{H}^3)$ to find new characterizations of quasi-convex (i.e. convex cocompact) subgroups of $PSL_2(\mathbb{C})$.
\end{abstract}

\tableofcontents

\section{Introduction}

Let $\Gamma$ be a hyperbolic group and $Isom^+(\mathbb{H}^3)\cong PSL_2(\mathbb{C})$ the group of orientation-preserving isometries of the 3-dimensional hyperbolic space. The group $Aut(\Gamma)$ naturally acts on the set of representations $Hom(\Gamma,PSL_2(\mathbb{C}))$ by precomposition and induces an action of the group $Out(\Gamma)$ on the set of conjugacy classes of representations $X(\Gamma,PSL_2(\mathbb{C}))=Hom(\Gamma,PSL_2(\mathbb{C}))\diagup PSL_2(\mathbb{C})$ where $PSL_2(\mathbb{C})$ is understood to act by inner automorphisms. We are interested in the dynamics of this action on invariant subsets, for example, it is well-known (see e.g. \cite{Canary}) that the set of conjugacy classes of convex cocompact representations is an open subset, invariant by the action of $Out(\Gamma)$, on which the action is properly discontinous. If $\rho$ is a representation, an associated orbit map is a $\rho$-equivariant map between a Cayley graph for $\Gamma$, denoted $C(\Gamma)$, and $\mathbb{H}^3$. One way to prove the previous statement is to use the fact that (conjugacy classes of) convex-cocompact representations are exactly those whose orbit maps are quasi-isometric embeddings, i.e., sends all geodesics of $C(\Gamma)$ to uniform quasi-geodesics of $\mathbb{H}^3$. 

However, the dynamics of the action of $Out(\Gamma)$ on the complementary set of conjugacy classes of convex cocompact representations is mostly unknown and depends on the choice of the hyperbolic group $\Gamma$ :\\

On the one hand, when $\Gamma \cong \pi_1(S_g)$, Goldman (see \cite{Goldman}) conjectured that the action of $Out(\pi_1(S_g))$ on the complementary set of conjugacy classes of convex cocompact representations $X(\pi_1(S_g),PSL_2(\mathbb{C}))\setminus CC(\pi_1(S_g),PSL_2(\mathbb{C}))$ is ergodic.\\

On the other hand, when $\Gamma \cong F_k$ is a free group of rank $k\geqslant 2$, Minsky \cite[Theorem 1.1]{Minsky} proved that there is an open set of conjugacy classes of representations of $F_k$ into $PSL_2(\mathbb{C})$, called the set of primitive-stable representations, which is invariant by the action of $Out(F_k)$, strictly larger than the set $CC(F_k,PSL_2(\mathbb{C}))$ and on which $Out(F_k)$ acts properly dicontinously.\\

In this article we will work with a broader framework, from now on, $\Gamma$ will denote a hyperbolic group and $G=Isom(X)$ will denote the isometry group of a Gromov-hyperbolic, proper and geodesic metric space $X$.\\

In that context, quasi-convex representations naturally generalize convex cocompact representations in the sense that they are precisely those whose orbit maps are quasi-isometric embeddings (see e.g. \cite[Corollary 1.8.4]{Bourdon}), hence if we denote by $QC(\Gamma,G)$ the set of conjugacy classes of quasi-convex representations, we can show that $QC(\Gamma,G)$ is an open and $Out(\Gamma)$-invariant subset on which the action is properly discontinuous.\\

Relying on the work of Minsky in \cite{Minsky} and Lee in \cite{Lee}, we generalize the definition of a primitive-stable representation. Starting with an $Aut(\Gamma)$-invariant subset $A\subset \Gamma$, one can define the set of \textit{$A$-stable representations}. Roughly speaking, an \textit{$A$-stable representation} is a representation whose orbit map sends all geodesics corresponding to infinite order elements of $A$ in $C(\Gamma)$, to uniform quasi-geodesics of $X$.\\
We denote by $S_A(\Gamma,G)$ the set of conjugacy classes of $A$-stable representations.

\begin{prop}\label{first}
If $A\subset \Gamma$ is invariant by automorphisms, then $S_A(\Gamma,G)$ is an open $Out(\Gamma)$-invariant subset of $X(\Gamma,G)$ which contains (strictly or not) $QC(\Gamma,G)$.
\end{prop}

Generalizing the work of the authors in \cite{Well-displacing}, we define the set of $A$-well-displacing representations. If $g$ is an isometry of a metric space $X$, the \textit{translation length} of $g$ is $l(g)= \inf_{x\in X} d(x, g(x))$. Recall that $\Gamma$ acts isometrically on $C_S(\Gamma)$, we denote by $l_S(\gamma)$ the translation length of $\gamma\in \Gamma$ for this action. 

If $A\subset \Gamma$ is $Aut(\Gamma)$-invariant, a representation $\rho$ of $\Gamma$ in an isometry group is said to be \textit{$A$-well-displacing} if there are $J\geqslant 1$ and $B\geqslant 0$ such that :
$$\frac{1}{J}l_S(a)-B \leqslant l(\rho(a)) \leqslant Jl_S(a)+B ~~~~ \forall a \in A.$$

It can be shown that an $A$-stable representation is always $A$-well-displacing and that the converse holds when $A=\Gamma$ (see \cite[Sections 3 and 4]{Well-displacing}). In the following, we prove :

\begin{thm}\label{Awd}
If $A$ is a characteristic subgroup of $\Gamma$, a representation $\rho \in Hom(\Gamma,G)$ is $A$-stable if and only if it is $A$-well-displacing.
\end{thm}

As a by-product of our work, we also establish the following general criterion for a representation to have a discrete image and a finite kernel.

\begin{prop}\label{criterion}
Let $\Gamma$ be a hyperbolic group, $\mathcal{C}$ denotes the subset of commutators of $\Gamma$ and $G$ be any isometry group endowed with a suitable topology. If a representation $\rho\in R(\Gamma,G)$ is $C$-well-displacing, then $\rho$ has discrete image and finite kernel.
\end{prop}

Focusing on the case $G=PSL_2(\mathbb{C})$, we also find new characterizations of convex cocompact representations :

\begin{thm}\label{df2}
Let $\Gamma$ be a non-elementary hyperbolic group, $\rho\in Hom(\Gamma,PSL_2(\mathbb{C}))$ be a representation, $\mathcal{C}$ denotes the subset of commutators of $\Gamma$ and $\mathcal{D}=[\Gamma,\Gamma]$ denotes the derived subgroup of $\Gamma$. The following are equivalent :
\begin{itemize}
\item The representation $\rho$ is convex cocompact,
\item The representation $\rho$ is $\mathcal{C}$-stable,
\item The representation $\rho$ is $\mathcal{D}$-well-displacing.
\end{itemize}
\end{thm}

Going back to the more general framework and focusing on dynamics, we exhibit a sufficient condition on the $Aut(\Gamma)$-invariant subset $A\subset \Gamma$ for the action of $Out(\Gamma)$ on the set $S_A(\Gamma,G)$ to be properly discontinuous.

\begin{defin}[Test subset]
A subset $A\subset \Gamma$ is a \textit{test subset} if, for all sequences $(\phi_n)_n$ of distinct elements of $Out(\Gamma)$,  there is $a\in A$ such that $\limsup_n~ l_S(\phi_n(a)) = \infty$.
\end{defin}

\begin{thm}\label{Apd}
If $A\subset \Gamma$ is $Aut(\Gamma)$-invariant and is a test subset, then $Out(\Gamma)$ acts properly discontinuously on $S_A(\Gamma,G)$.
\end{thm}

In investigating whether certain subsets of $\Gamma$ are test subsets, we answered in the affirmative a question asked by Canary in the appendix of \cite{Canary}, namely that the set of generators and products of distinct generators of a hyperbolic group is always a test subset.

\begin{thm}\label{Can}
For all hyperbolic groups $\Gamma$, for all finite generating set $S$ and all $M\geqslant 0$, the set 
$$\{\phi\in Out(\Gamma)  ~|~ l_S(\phi(b))\leqslant M ~~ \forall b\in \B_S\}$$
is finite where $\B_S= S \cup \{ss'~|~ s,s'\in S$, $s\ne s' \}$ is the set of elements and product of distinct elements of $S$.
\end{thm}

We now briefly outline the rest of this article. In Section 2, we recall some preliminary facts and theorems about hyperbolic spaces and their isometry group. In Section 3, we present the set $X(\Gamma,G)$ of conjugacy classes of representations of a group $\Gamma$ in an isometry group $G$ and describe the natural action of $Out(\Gamma)$ on $X(\Gamma,G)$. Moreover, we define some classical invariant subsets of $X(\Gamma,G)$ and generalize a theorem of Sullivan. In Section 4, we construct new $Out(\Gamma)$-invariant subsets of $X(\Gamma,G)$ using some $Aut(\Gamma)$-invariant subset of $\Gamma$ and prove Proposition \ref{first} as well as Theorems \ref{Apd} and \ref{Awd}. In Section 5, we recall some theory on $\mathbb{R}$-trees and prove Theorem \ref{Can}. Finally, in last section, we recall some basic facts about convergence groups and use it to prove Proposition \ref{criterion} and Theorem \ref{df2}. In the appendix, we show that the derived subgroups of free groups of rank bigger than 3 contains a finite test subset.

\section{Hyperbolic spaces and their isometry groups}

Good references for this section are \cite{Quasi-geo} or in \cite{Ghys dlH}.

\subsection{Hyperbolic spaces and their compactification}

Let $(X,d)$ be a metric space and $x,y,w\in X$. The \textit{Gromov product} of $x$ and $y$ based at $w$ and denoted by $\langle x,y \rangle_w$ is defined by 

$$\langle x,y \rangle_w=\frac{1}{2}(d(x,w)+d(y,w)-d(x,z))$$

A metric space $X$ is \textit{Gromov-hyperbolic} if there is $\delta>0$ such that for all $w,x,y,z$ in $X$, we have 

$$\langle x,z \rangle_w \geqslant \min\{\langle x,y \rangle_w,\langle y,z \rangle_w\} - \delta$$

If a metric space $X$ is Gromov-hyperbolic, then we call such a $\delta$ a \textit{hyperbolicity constant} for $X$.\\

If $X$ is a Gromov-hyperbolic and geodesic metric space, we denote by $\partial X$ the Gromov boundary of $X$ and $\overline{X}= X\cup \partial X$ its Gromov compactification.

\subsection{Isometries of a hyperbolic space}

Let $G=Isom(X)$ be the isometry group of a Gromov-hyperbolic and geodesic metric space $X$. Let $g\in G$ and $o\in X$.\\

We define two invariants of conjugacy classes of elements of $G$ which will be helpful to get a classification.\\

The \textit{translation length of $g$} denoted $l(g)$, is defined by $l(g)=\inf_{x\in X} d(x,g(x))$.\\
Sometimes, it will be more convenient to work with the \textit{stable norm of $g$} denoted $N(g)$ and defined by $N(g)=\lim_{n\to +\infty}\frac{d(o,g^n(o))}{n}$.\\

One can check that the stable norm is well-defined by a subadditivity argument and that the limit does not depend on the choice of the base point $o\in X$.

The next proposition shows that, in our context, these two invariants only differ by an additive constant. Its proof can be found in \cite[Proposition 6.4]{Quasi-geo} :

\begin{prop}\label{translation}
If $X$ is a $\delta$-hyperbolic and geodesic metric space and $g\in Isom(X)$ is an isometry of $X$, then
\begin{equation*}
N(g)\leqslant l(g) \leqslant N(g)+16\delta
\end{equation*}
\end{prop}

Recall that an isometry of $X$ extends to a homeomorphism of its Gromov compactification $\overline{X}$. Elements of $G$ fall into one of the three following categories :  

\begin{itemize}
\item We say that $g$ is \textit{elliptic} if $N(g)=0$ and $g$ has a fixed point in $X$.
\item We say that $g$ is \textit{parabolic} if $N(g)=0$ and $g$ is not elliptic. In that case, the extended action of $g$ has a unique fixed point in $\partial X$.
\item We say that $g$ is \textit{hyperbolic} if $N(g)>0$. In that case, the extended action of $g$ has two distinct fixed points in $\partial X$.
\end{itemize}

\subsection{Subgroups of an isometry group}

Let $G=Isom(X)$ be the isometry group of a Gromov-hyperbolic, proper (closed balls are compact) and geodesic metric space $X$.

\begin{defi}\label{qua-con}
We say that $Z\subset X $ is \textit{quasi-convex} if there is $\eta>0$ such that any two points of $Z$ can be connected by a geodesic segment contained in an $\eta$-neighborhood of $Z$.\\
A subgroup $H$ of $G$ is \textit{quasi-convex} if it is discrete and if there is a point $x\in X$ such that the orbit $H(x)$ is quasi-convex.
\end{defi}

Let $G'$ be a semi-simple rank one Lie group and $X'$ is its associated symmetric space (e.g. $G'=SO_0(n,1)$ and $X'=\mathbb{H}^n$), recall that any discrete torsion-free subgroup $H$ of $G'$ yields a manifold $X'\diagup H$ with fundamental group isomorphic to $H$.

\begin{defi}\label{con-coc}
A discrete subgroup $H$ of $G'$ is \textit{convex cocompact} if there is a convex and $H$-invariant subset of $X'$ on which the action of $H$ is cocompact.
\end{defi}

If $H$ is a convex cocompact subgroup of $G'$, then there is a convex and compact submanifold $\Omega$ of $X'\diagup H$ such that the inclusion $\Omega \hookrightarrow X' \diagup H$ is a homotopy equivalence.\\

Actually, the two previous properties are equivalent (see, e.g., \cite[Section 1.8]{Bourdon}) :

\begin{prop}
If $G$ is a semi-simple rank one Lie group, a subgroup of $G$ is convex cocompact if and only if it is quasi-convex. 
\end{prop}

\subsection{Hyperbolic groups}

Let $\Gamma$ be a finitely generated group and $S$ be a finite set of generators. The \textit{Cayley graph} of $\Gamma$ with respect to $S$, denoted $C_S(\Gamma)$, is the graph whose vertices are labelled by elements of $\Gamma$, and there is an edge between two vertices labelled by $\gamma,\gamma'\in \Gamma$ if there is $s\in S$ such that $\gamma'=s\gamma$.\\
The metric on $C_S(\Gamma)$ obtained by setting each edge to be of length 1, yields a metric on $\Gamma$ called \textit{word metric} and denoted by $d_S$.\\

A group $\Gamma$ is \textit{hyperbolic} if there is a finite generating set $S$ such that $(C_S(\Gamma),d_S)$ is a Gromov-hyperbolic metric space. A hyperbolic group is \textit{non-elementary} if it contains a free subgroup of rank 2.\\

The group $\Gamma$ acts naturally by isometries on $(C_S(\Gamma),d_S)$ by left multiplication on the vertices. Every element $\gamma\in \Gamma$ of infinite order acts as a hyperbolic isometry and every element $\gamma\in \Gamma$ of finite order acts as an elliptic isometry.\\

The \textit{word length of $\gamma\in \Gamma$} is $|\gamma|_S = d_S(e,\gamma)$. We write $l_S(\gamma)$ the translation length of $\gamma\in\Gamma$ for its action on $C_S(\Gamma)$ and one can check that $l_S(\gamma)=\min_{g\in\Gamma} |g\gamma g^{-1}|_S$. 

We will need the following proposition (for a proof, see \cite[Corollary p.224]{Ghys}).

\begin{prop}\label{inf}
If $\Gamma$ is a hyperbolic group and $H$ is a subgroup of $\Gamma$, $H$ is infinite if and only if it contains an infinite order element.
\end{prop}

\begin{defi}
Let $(X,d)$ and $(Y,d')$ be two metric spaces.

\begin{enumerate}
\item Let $K\geqslant 1$ and $C\geqslant 0$ be two constants. A map $f:X \longrightarrow Y$ is a \textit{($K,C$) quasi-isometric embedding} if for all $x,z\in X$ : 
\begin{center}
$\frac{1}{K}d(x,z)-C \leqslant d'(f(x),f(z)) \leqslant Kd(x,z)+C$
\end{center}

\item A map $f:X \longrightarrow Y$ is a \textit{quasi-isometric embedding} if there are $K\geqslant 1$ and $C\geqslant 0$ such that $f$ is a ($K,C$) quasi-isometric embedding.

\item Let $K\geqslant 1$ and $C\geqslant 0$ be two constants. A map $f:X \longrightarrow Y$ is a \textit{($K,C$) quasi-isometry} if $f$ is a ($K,C$) quasi-isometric embedding and if there is a ($K,C$) quasi-isometric embedding $g:Y \longrightarrow X$ such that $d(g\circ f (x),x) \leqslant C$ and $d'(f\circ g (y),y) \leqslant C$ for all $x\in X$, $y\in Y$.

\item A map $f:X \longrightarrow Y$ is a \textit{quasi-isometry} if there are $K\geqslant 1$ and $C\geqslant 0$ such that $f$ is a ($K,C$) quasi-isometry.

\item A \textit{$(K,C)$ quasi-geodesic} of $X$ is the image of a $(K,C)$ quasi-isometric embedding $f:\mathbb{Z} \longrightarrow X$ or $f:\mathbb{R}\longrightarrow X$.

\end{enumerate}

\end{defi}

If $\gamma\in \Gamma$ has infinite order, a \textit{$(K,C)$ quasi-axis for $\gamma$} is a $(K,C)$ quasi-geodesic of $C_S(\Gamma)$, invariant under the action of $\gamma$ and such that $\gamma$ acts on it by translation of length $l_S(\gamma)$. The endpoints of a quasi-axis of for $\gamma$ are the distinct fixed points in $\partial C_S(\Gamma)$ for the extended action of $\gamma$ on $\overline{C_S(\Gamma)}$, \\

There are uniform constants $(K',C')$ such that every infinite order element of $\Gamma$ has a $(K',C')$ quasi-axis for its action on $C_S(\Gamma)$ (see e.g. \cite[Lemma II.6]{Lee} for a proof of the following result). More precisely :

\begin{lem}\label{l1}
Let $\Gamma$ be a hyperbolic group, $S$ be a finite set of generators, $\delta_S$ be a hyperbolicity constant of $C_S(\Gamma)$ and $\gamma\in \Gamma$ with infinite order.\\
Let $w_\gamma\in C_S(\Gamma)$ be such that $d_S(w_\gamma, \gamma \cdot w_\gamma)=l_S(\gamma)$ and let $l_\gamma:=\bigcup_{n\in\mathbb{Z}}\gamma^n \cdot [w_\gamma,\gamma \cdot w_\gamma]$ where $[w_\gamma,\gamma \cdot w_\gamma]$ denotes any geodesic segment connecting $w_\gamma$ to $\gamma \cdot w_\gamma$ in $C_S(\Gamma)$.\\
There are constants $K'\geqslant 1$ and $C'\geqslant 0$ depending only on $\delta_S$ such that for all $\gamma\in \Gamma$ of infinite order, and for any choice of $w_\gamma\in C_S(\Gamma)$ as above, $l_\gamma$ is a $(K',C')$ quasi-axis for $\gamma$.
\end{lem}

\begin{rems}
Let $\Gamma$ be a hyperbolic group and $K',C'$ the uniform constants given by the previous lemma.\\
If $\gamma$ has minimal word length in its conjugacy class, then $\bigcup_{n\in \mathbb{Z}}\{\gamma^n\}$ is a $(K',C')$ quasi-axis for $\gamma$.\\
If $g,h\in \Gamma$ and $l_g$ is a $(K',C')$ quasi-axis for $g$, then $h \cdot l_g$ is a $(K',C')$ quasi-axis for $hgh^{-1}$. 
\end{rems}

\section{Space of representations and open sets of discontinuity}

Let $X$ be a Gromov-hyperbolic, proper and geodesic metric space and let $G=Isom(X)$ be its isometry group endowed with the compact-open topology. Let $\Gamma$ be a finitely generated group and let $S$ be a finite set of generators.\\

In this section, we define the space $X(\Gamma,G)$ of conjugacy classes of representations of $\Gamma$ in $G$. The group $Out(\Gamma)$ of outer automorphism of $\Gamma$ acts naturally on $X(\Gamma,G)$ by precomposition. We define $Out(\Gamma)$-invariant subsets of $X(\Gamma,G)$ and look at the dynamical properties of the action on those invariant subsets.

\subsection{Group action on the space of representations}

The \textit{space of representations}, $R(\Gamma,G)$ is the set of homomorphisms $\rho:\Gamma\longrightarrow G$. We identify $R(\Gamma,G)$ with a subset of $G^{|S|}$, which allows us to endow $R(\Gamma,G)$ with the product topology. This topology does not depend on the choice of a finite generating set $S$.\\
The automorphism groups of $G$ and of $\Gamma$ act on $R(\Gamma,G)$ in the following way : if $\rho\in R(\Gamma,G)$, $\psi\in Aut(G)$ and $\varphi\in Aut(\Gamma)$,
\begin{center}
$\psi \cdot \rho := \psi \circ \rho \qquad $ and $ \qquad \varphi \cdot \rho := \rho \circ \varphi^{-1}$
\end{center}
For any group $G$, $Inn(G)< Aut(G)$ will denote the normal subgroup of inner automorphisms of $G$ and  $Out(G)=Aut(G)\diagup Inn(G)$ the quotient group of outer automorphisms.\ \
We denote by $X(\Gamma,G)=R(\Gamma,G)\diagup Inn(G)$ the quotient space made of conjugacy classes of representations, and endow it with quotient topology. We mention that this topology may not be Hausdorff (may not even be a $T1$-space), however, this fact will not add any difficulty in what follows.
The group $Inn(\Gamma)$ acts trivially on $X(\Gamma,G)$ and this induces an action of $Out(\Gamma)$ on $X(\Gamma,G)$.\\

We denote by $AH(\Gamma,G)$ the set of conjugacy classes of representations whose image is discrete in $G$ and whose kernel is finite. Remark that if $\Gamma$ is torsion-free, $[\rho]\in AH(\Gamma,G)$ if and only if it has a representative which is injective and has discrete image in $G$.\\

Using Definitions \ref{con-coc} and \ref{qua-con}, we define the following subsets of $AH(\Gamma,G)$ :

\begin{defi}
We say that $\rho\in R(\Gamma,G)$ is \textit{quasi-convex} if $\rho$ has finite kernel and its image is a quasi-convex subgroup of $G$.\\
If $G$ is a semi-simple rank one Lie group, we say that $\rho\in R(\Gamma,G)$ is \textit{convex cocompact} if $\rho$ has finite kernel and its image is a convex cocompact subgroup of $G$.\\
\end{defi}

We denote by $QC(\Gamma,G)$ (resp. $CC(\Gamma,G)$) the set of conjugacy classes of quasi-convex representations (resp. the set of conjugacy classes of convex cocompact representations whenever $G$ is a semi-simple rank one Lie group).\\

One can check that $AH(\Gamma,G)$, $QC(\Gamma,G)$ and $CC(\Gamma,G)$ are invariant under the action of the outer automorphism group of $\Gamma$.

\subsection{Orbit maps and quasi-isometric embeddings}

Let $\rho\in R(\Gamma,G)$ and $x\in X$ be a base point.\\

An \textit{orbit map} $\tau_{\rho,x}:C_S(\Gamma) \longrightarrow X$ is a $\rho$-equivariant map defined by $\tau_{\rho,x}(\gamma)= \rho(\gamma)(x)$ for all $\gamma\in \Gamma$, and sends each edge $[e,s]$ of $C_S(\Gamma)$ on a geodesic segment of $X$ connecting $x$ to $s \cdot x$ for all $s\in S$.

\begin{rem}\label{lipschitz}
Let $\rho \in R(\Gamma,G)$, $\tau_{\rho,x}$ be an orbit map, and $\kappa:=\max_{s\in S}\{d(x, \tau_{\rho,x}(s))\}$.\\
Note that $\tau_{\rho,x}$ is $\kappa$-Lipschitz.
\end{rem}

Next proposition gives a characterization of quasi-convex representations via the orbit map. A proof can be found in \cite[Corollary 1.8.4]{Bourdon}.

\begin{prop}\label{qc}
A representation $\rho\in R(\Gamma,G)$ is quasi-convex if and only if there are a point $x\in X$ and a finite set $S$ of generators for $\Gamma$ such that there is an orbit map $\tau_{\rho,x}:C_S(\Gamma) \longrightarrow X$ which is a quasi-isometric embedding.
\end{prop}

\begin{cor}
If $QC(\Gamma,G)\ne \emptyset$, then $\Gamma$ is a hyperbolic group.
\end{cor}

\begin{proof}
If $\rho\in QC(\Gamma,G)$, then $\Gamma$ acts by isometries, cocompactly and properly discontinuously on the convex hull of $\Or(\Gamma)$ which is Gromov-hyperbolic. Indeed, a convex subset of a Gromov-hyperbolic metric space is also Gromov-hyperbolic, and the action of $\Gamma$ on $\Or(\Gamma)$ being cocompact, its action on the convex hull of $\Or(\Gamma)$, which is at bounded Hausdorff distance from $\Or(\Gamma)$, is also cocompact. \v{S}varc-Milnor lemma implies that $\Gamma$ is a hyperbolic group.
\end{proof}

\subsection{Well-displacing representations and open sets of discontinuity}

Recall that we denote by $l_S(\gamma)$ the translation length of $\gamma$ for its action on $C_S(\Gamma)$ and $l(\rho(\gamma))$ the translation length of $\rho(\gamma)$ for its action on $X$. 

We say that $\rho\in R(\Gamma,G)$ is \textit{well-displacing} if there are $J\geqslant 1$ and $B\geqslant 0$ such that :
$$\frac{1}{J}l_S(\gamma)-B \leqslant l(\rho(\gamma)) \leqslant Jl_S(\gamma)+B ~~~~ \forall \gamma\in \Gamma.$$

This property is studied in \cite[Sections 3 and 4]{Well-displacing} where the authors showed that :

\begin{prop}\label{wd}
Let $\Gamma$ be a hyperbolic group and $\rho\in R(\Gamma,G)$. A representation $\rho$ is quasi-convex if and only if it is well-displacing.
\end{prop}

We will sketch the proof of this proposition in greater generality in Section 4.5.\\

Let $H$ be a group acting continuously on a topological space $X$ and $Y\subset X$ be invariant by the action of $H$. Recall that $H$ acts \textit{properly discontinuoulsy} on $Y$ if for every compact set $R\subset Y$, the set\\ $\{h\in H ~|~ h(R)\cap R \ne \emptyset\}$ is finite.
We say that $Y$ is an \textit{open set of discontinuity} if $Y$ is open, $H$-invariant, and that the action of $H$ on $Y$ is properly discontinuous.

The next proposition (see Section 6.1 for a proof) can be seen as a special case of the work done in the following : 

\begin{prop}\label{QC}
Let $\Gamma$ be a hyperbolic group. The set $QC(\Gamma,G)$ is an open set of discontinuity for the action of $Out(\Gamma)$ on $X(\Gamma,G)$.
\end{prop}

When $G=Isom^+(\mathbb{H}^3) \cong PSL(2,\mathbb{C})\cong SO_0(3,1)$, Sullivan (see \cite[Theorem A, Section 7]{Sullivan}) proved that the set of conjugacy classes of quasi-convex (i.e. convex cocompact) representations of a torsion-free hyperbolic group is precisely the interior of the set of conjugacy classes of discrete and faithful representations. We generalize this result by removing the torsion-free assumption, for this purpose, we will need the following lemmas :

\begin{lem}\label{Peter}
If $\Gamma$ is a finitely generated group such that $AH(\Gamma,PSL(2,\mathbb{C}))\ne \emptyset$, then $\Gamma$ contains a torsion-free subgroup of finite index.
\end{lem}

\begin{proof}
Let $\rho\in AH(\Gamma,PSL(2,\mathbb{C}))$ a representation of $\Gamma$ in $PSL(2,\mathbb{C})$ with finite kernel and discrete image. The group $\rho(\Gamma)$ is a finitely generated linear group hence contains a torsion-free subgroup of finite index according to Selberg's lemma (see \cite{Selberg}), that we will denote by $H_0$. The group $H_0$ is isomorphic to the fundamental group of a 3-manifold that we can suppose to be compact wlog, according to the tameness theorem (see \cite{Agol}, \cite{CG}). It follows from \cite[Theorem 1.3]{Hai} that the groups $\Gamma$ and $\rho(\Gamma)$ are commensurable, i.e. contains isomorphic finite index subgroups that we will denote $Q<\Gamma$ and $H<\rho(\Gamma)$. Let $\psi: Q \rightarrow H$ denote the previous isomorphism, then $\psi^{-1}(H\cap H_0)$ is a torsion-free subgroup of $\Gamma$ with finite index.
\end{proof}

The following lemma is a particular case of \cite[Corollary 1.3]{GW}. The reader is referred to \cite[Corollary 3.4 and Theorem 5.15]{GW} for a proof.

\begin{lem}\label{Gui}
Let $\Gamma$ be a hyperbolic group and $\Gamma'$ be a finite index subgroup of $\Gamma$. A representation $\rho \in R(\Gamma,PSL(2,\mathbb{C}))$ is convex-cocompact if and only if the restriction of $\rho$ to $\Gamma'$ is convex-cocompact.
\end{lem}

\begin{thm}\label{Sullivan}
Let $\Gamma$ be a hyperbolic group.
The set $CC(\Gamma,PSL_2(\mathbb{C}))$ of conjugacy classes of convex cocompact representations is precisely the interior of $AH(\Gamma,PSL_2(\mathbb{C}))$, the set of conjugacy classes of representations with discrete image and finite kernel.
\end{thm}

\begin{proof}
Suppose that the set $AH(\Gamma,G)$ of conjugacy classes of representations with discrete image and finite kernel is not empty. We already know that $CC(\Gamma,G)\subset Int(AH(\Gamma,G))$ (see Proposition \ref{QC}), we will prove the other inclusion using contraposition. Let $[\rho]\in AH(\Gamma,G)\setminus CC(\Gamma,G)$ and $U$ be a neighbourhood of $[\rho]$ we will prove that there is $[\sigma]\in U$ such that $[\sigma]\notin AH(\Gamma,G)$. We denote by $\Gamma_0$ the torsion-free finite index subgroup of $\Gamma$ given by Lemma \ref{Peter} and by $r: X(\Gamma,G) \rightarrow X(\Gamma_0,G)$ the restriction map defined by $r([\sigma])=[\sigma_{|\Gamma_0}]$ for all $[\sigma]\in X(\Gamma,G)$. With an appropriate choice of generators for $\Gamma_0$ and $\Gamma$, we can show that $r$ is an open map. Indeed, write for some $g_1,...,g_k \in \Gamma$ : 
$$\Gamma=\Gamma_0\sqcup g_1\Gamma_0 \sqcup ... \sqcup g_k\Gamma_0$$
a partition of $\Gamma$ in $\Gamma_0$ left cosets. Consider $S_0$ a finite set of generators for $\Gamma_0$, and $S:=S_0 \cup \{g_1,...,g_k\}$. The finite set $S$ generates $\Gamma$ and the restriction map $r' : G^{|S|} ~ \rightarrow G^{|S_0|}$ is just the projection on the $|S_0|$ first coordinates. It follows that $r$ is an open map, hence if $U$ denotes a neighbourhood of $[\rho]$ in $X(\Gamma,G)$, its image $U_0:= r(U)$ is a neighbourhood of $[\rho_0]:= r([\rho])$. Lemma \ref{Gui} assures us that the conjugacy class $[\rho_0] \in AH(\Gamma_0,G)\setminus CC(\Gamma_0,G)$, hence Sullivan's theorem \cite[Theorem A, Section 7]{Sullivan} prove the existence of a representation $[\sigma_0]\in U_0$ such that $[\sigma_0] \notin AH(\Gamma_0,G)$ which in turn prove the existence of a representation $[\sigma]\in U$ such that $[\sigma] \notin AH(\Gamma,G)$.
\end{proof}

This theorem will be used in Section 6 to show Theorem \ref{df2}.

\section{$A$-stable representations}

In this section, let $\Gamma$ be a hyperbolic group, $S$ be a finite set of generators for $\Gamma$ and $G=Isom(X)$ be the isometry group of a Gromov-hyperbolic, proper and geodesic metric space $X$.\\

For any $Aut(\Gamma)$-invariant subset $A\subset \Gamma$, we define an $Out(\Gamma)$-invariant open subset $S_A(\Gamma,G)\subset X(\Gamma,G)$ called the set of $A$-stable representations. We give a sufficient condition on $A\subset \Gamma$ for the action of $Out(\Gamma)$ on $S_A(\Gamma,G)$ to be properly discontinuous and we continue by proving the equivalence between $A$-stability and the $A$-well-displacing property in the particular case when $A$ is a characteristic subgroup. We end this section by giving a geometric characterization, due to Lee in \cite{Lee}, of discrete, faithful and $A$-stable representations when $G=PSL_2(\mathbb{C})$.

The letter $A$ will always denote an $Aut(\Gamma)$-invariant subset of $\Gamma$.\\

\subsection{Definition}

If $\gamma\in \Gamma$ has infinite order, $\gamma_-$ and $\gamma_+$ will denote its fixed points in $\partial C_S(\Gamma)$ for the extended action of $\gamma$ on $\overline{C_S(\Gamma)}$. Moreover we denote by $L_S(\gamma)$ the set of geodesics of $C_S(\Gamma)$ connecting $\gamma_-$ to $\gamma_+$.\\
If $A\subset \Gamma$, we name $A_{\infty}$ the subset consisting of infinite order elements of $A$ and $L_S(A):=\bigcup_{a\in A_\infty} L_S(a)$.\\

\begin{defi} 
Let $A \subset \Gamma$ be a subset invariant by automorphism of $\Gamma$, $K\geqslant 1$ and $C\geqslant 0$ be two constants.\\ 
A representation $\rho\in R(\Gamma,G)$ is \textit{($K,C$) $A$-stable} if there is a point $x\in X$,  such that for all $l \in L_S(A)$, $\Or(l)$ is a ($K,C$) quasi-geodesic.\\
A representation $\rho\in R(\Gamma,G)$ is \textit{$A$-stable} if there are constants $K\geqslant 1,C\geqslant 0$ such that $\rho$ is ($K,C$) $A$-stable.
\end{defi}

It is easily checked that being $A$-stable depends neither on the conjugacy class of the representation, neither on the choice of the base point $x\in X$, nor on the choice of a finite set of generators for $\Gamma$.

We write $S_A(\Gamma,G)$ the set of conjugacy classes of $A$-stable representations.

\subsection{$S_A(\Gamma,G)$ is an invariant open set}

We now aim at proving the following proposition.

\begin{prop}\label{ouvert}
The set $S_A(\Gamma,G)$ is $Out(\Gamma)$-invariant and open in $X(\Gamma,G)$. \\
More precisely, if $\delta$ is a hyperbolicity constant for $X$, for all constants $K_0\geqslant 1$, $C_0\geqslant 0$, there are constants $K\geqslant 1$, $C\geqslant 0$ depending only on $K_0, C_0$ and $\delta$ such that for all $(K_0,C_0)$ $A$-stable representations $\rho_0$, there is an open neighbourhood $U_0$ of $\rho_0$ such that for all $\rho\in U_0$, $\rho$ is $(K,C)$ $A$-stable.
\end{prop}

In order to prove this proposition, we introduce :

\begin{defi}
Let $K\geqslant 1$, $C\geqslant 0$ and $M\geqslant 0$.\\
A \textit{($K,C,M$) local quasi-geodesic} is the image of $f:\mathbb{Z}\rightarrow X$ or $\mathbb{R}\rightarrow X$ such that $f$ restricted to any segment of length at most $M$ is a ($K,C$) quasi-isometric embedding.
\end{defi}

Next proposition is proved in \cite[Theorem 1.4]{Quasi-geo} and gives us a sufficient condition for a family of paths in a Gromov-hyperbolic metric space to be uniform quasi-geodesics.

\begin{prop}\label{local}
Let $X$ be a $\delta$-hyperbolic, geodesic metric space.\\
For all $K'\geqslant 1$ and $C'\geqslant 0$ there are  $K\geqslant 1$, $C\geqslant 0$ and $M\geqslant 0$ depending only on $\delta$, $K'$ and $C'$ such that any ($K',C',M$) local quasi-geodesic of $X$ is a ($K,C$) quasi-geodesic.\\
\end{prop}

\begin{proof}[Proof of Proposition \ref{ouvert}]
We start by proving that $S_A(\Gamma,G)$ is $Out(\Gamma)$-invariant. Let $\rho\in R(\Gamma,G)$ be an $A$-stable representation and $\varphi\in Aut(\Gamma)$. By definition, $\tau_{\rho,x}$ maps all geodesics of $L_S(A)$ to uniform quasi-geodesics of $X$. This implies that the orbit map $\tau_{\varphi\cdot \rho, x} = \tau_{\rho \circ \varphi^{-1},x}$ maps all geodesics of $L_{\varphi(S)}(A)$ to uniform quasi-geodesics of $X$. Since $\varphi(S)$ is a set of generators for $\Gamma$ and $A$-stability does not depend on the choice of generators, $\varphi\cdot \rho$ is also $A$-stable.\\

Let us now prove that $S_A(\Gamma,G)\subset X(\Gamma,G)$ is open in $X(\Gamma,G)$.\\
Let $[\rho_0]\in S_A(\Gamma,G)$ be the conjugacy class of a $(K_0,C_0)$ $A$-stable representation $\rho_0$.\\
Let $M,K,C$ be the constants obtained by Proposition \ref{local} associated to $\delta$, $K_0$ and $2C_0$.\\
We choose an open neighbourhood $U_0$ of $[\rho_0]$ in $X(\Gamma,G)$ such that for any conjugacy class of representations in $U_0$, there is a representative $\rho$ such that $d(\tau_{\rho,x}(\gamma),\tau_{\rho_0,x}(\gamma))\leqslant C_0$ for all $\gamma\in \overline{B_S(\mathrm{e},M)}$ in a closed ball of radius $M$ around the neutral element.\\

We will show that every conjugacy class of representations in $U_0$ is ($K,C$) $A$-stable. Proposition \ref{local} ensures us that it is sufficient to show that for all elements in $U_0$, there is a representative representation $\rho$ such that for every geodesic $l\in L_S(A)$ , $\tau_{\rho,x}(l)$ is a ($K_0,2C_0,M$) local quasi-geodesic of $X$.
Let $l=\{l_n\}_{n\in\mathbb{Z}}$ be a geodesic of $L_S(A)$, i.e, $d_S(l_n,l_m)=|n-m|$ for all $m,n\in \mathbb{Z}$. 

Let $n,m\in \mathbb{Z}$ such that $|n-m|\leqslant M$ :

\begin{equation*} 
\begin{split}
d(\tau_{\rho,x}(l_n),\tau_{\rho,x}(l_m)) & = d(\tau_{\rho,x}(e),\tau_{\rho,x}(l_n^{-1}l_m)) \\
 & \leqslant d(x,\tau_{\rho_0,x}(l_n^{-1}l_m)) + d(\tau_{\rho_0,x}(l_n^{-1}l_m), \tau_{\rho,x}(l_n^{-1}l_m)) \\
 & \leqslant d(\tau_{\rho_0,x}(l_n), \tau_{\rho_0,x}(l_m)) + C_0 \\
 & \leqslant K_0|n-m|+2C_0
\end{split}
\end{equation*}

Moreover, by symmetry
\begin{equation*} 
d(\tau_{\rho_0,x}(l_n),\tau_{\rho_0,x}(l_m)) \leqslant d(\tau_{\rho,x}(l_n), \tau_{\rho,x}(l_m)) + C_0
\end{equation*}

We deduce that
\begin{equation*} 
\begin{split}
d(\tau_{\rho,x}(l_n),\tau_{\rho,x}(l_m)) & \geqslant d(\tau_{\rho_0,x}(l_n),\tau_{\rho_0,x}(l_m))-C_0 \\
 & \geqslant \frac{|n-m|}{K_0}-2C_0
\end{split}
\end{equation*}

The orbit map $\tau_{\rho,x}$ maps every geodesic $l\in L_S(A)$ to ($K_0,2C_0,M$) local quasi-geodesics of $X$, hence ($K,C$) quasi-geodesics of $X$ according to Proposition \ref{local} where the constants $K$ and $C$ depends neither on $[\rho]\in U_0$  nor on $l\in L_S(A)$.

\end{proof}

\subsection{$A$-well-displacing representations}

\begin{defi}
Let $A \subset \Gamma$ be a subset invariant by automorphism of $\Gamma$, $J\geqslant 1$ and $B\geqslant 0$ be two constants.\\
We say that a representation $\rho\in R(\Gamma,G)$ is \textit{$(J,B)$ $A$-well-displacing} if there are $J\geqslant 1$ and $B\geqslant 0$ such that :
$$\frac{1}{J}l_S(a)-B \leqslant l(\rho(a)) \leqslant Jl_S(a)+B ~~~~ \forall a\in A$$
We say that a representation $\rho\in R(\Gamma,G)$ is \textit{$A$-well-displacing} if there are $J\geqslant 1$ and $B\geqslant 0$ such that $\rho$ is $(J,B)$ $A$-well-displacing.
\end{defi}

\begin{prop}\label{displacing}
Given two constants $K\geqslant 1$ and $C\geqslant 0$, there are constants $J\geqslant 1$ and $B\geqslant 0$ such that every $(K,C)$ $A$-stable representation is $(J,B)$ $A$-well-displacing.
\end{prop}

In order to prove this proposition, we will use the following lemma proved in \cite[Lemma II.9]{Lee}. 

\begin{lem}\label{l2}
Let $\rho\in R(\Gamma,G)$ be a representation, let $x\in X$ be a basepoint, let $K,K'\geqslant 1$ and $C,C'\geqslant 0$ be some constants and let $l$ be a geodesic of $C_S(\Gamma)$.\\
There are constants $K''\geqslant 1$ and $C''\geqslant 0$ such that, if $\Or(l)$ is a $(K,C)$ quasi-geodesic of $X$, then for any $(K',C')$ quasi-geodesic $l'$ of $C_S(\Gamma)$ with the same endpoints than $l$, $\Or(l')$ is a $(K'',C'')$ quasi-geodesic of $X$.
\end{lem}

\begin{proof}[Proof of Proposition \ref{displacing}]
Let $\rho\in R(\Gamma,G)$ be a $(K,C)$ $A$-stable representation and $a\in A_\infty$.\\
Wlog, we can assume that $a$ is of minimal word length in its conjugacy class, hence $l_a=\{a^n\}$ is a $(K',C')$ quasi-geodesic of $C_S(\Gamma)$ using Lemma \ref{l1} with constants $K'$ and $C'$ independent of $a$. The quasi-geodesic $l_a$ has the same endpoints than some geodesic $l$ of $L_S(a)$.

Since $\Or(l)$ is a $(K,C)$ quasi-geodesic of $X$, Lemma \ref{l2} implies that there are constants $K''\geqslant 1$ and $C''\geqslant 0$ independent of $a$ such that $\Or(l_a)$ is $(K'',C'')$ quasi-geodesic of $X$.

Hence, for all $n\in \mathbb{N}^*$ and all $a\in A_\infty$ which have minimal word length in its conjugacy class :

\begin{equation*}
\begin{split}
\frac{1}{K''K'}|a^n|_S-\frac{C'}{K''}-C'' &\leqslant \frac{1}{K''}n-C'' \\
&\leqslant d(x,\rho(a^n)(x)) \\
&\leqslant K''n+C'' \\
& \leqslant K''K'|a^n|_S +K''C'+C''
\end{split}
\end{equation*}

If we denote by $N_S(\gamma)$ the stable norm of $\gamma\in \Gamma$ for its action on $C_S(\Gamma)$, dividing by $n$ and taking the limit yield :
\begin{equation*}
\frac{1}{K''K'} N_S(a) \leqslant N(\rho(a)) \leqslant K''K' N_S(a), ~~~~ \forall a\in A_\infty
\end{equation*}

Proposition \ref{translation} gives us constants $J\geqslant 1$ and $B\geqslant 0$ such that :
$$\frac{1}{J}l_S(a)-B \leqslant l(\rho(a)) \leqslant Jl_S(a)+B, ~~~~ \forall a\in A_\infty $$.\\

Since the image of a finite order element of $\Gamma$ by $\rho$ also has finite order in $G$, hence trivial translation length, we actually get :
$$\frac{1}{J}l_S(a)-B \leqslant l(\rho(a)) \leqslant Jl_S(a)+B, ~~~~ \forall a\in A.$$
\end{proof}

Observe that Proposition \ref{displacing} prove the direct implication of Proposition \ref{wd} by taking $A=\Gamma$.

Combining Propositions \ref{ouvert} and \ref{displacing}, we now show that :

\begin{cor}\label{displacing compact}
Let $R$ be a compact subset of $S_A(\Gamma,G)$, there are $J_R\geqslant 1$ and $B_R\geqslant 0$ such that every elements in $R$ is $(J_R,B_R)$ A-well-displacing.
\end{cor}

\begin{proof}
Proposition \ref{ouvert} implies that for every representation $\rho\in R$, there are $U_{\rho}$ an open neighbourhood of $\rho$ in $X(\Gamma,G)$, $K_{\rho}\geqslant 1$ and $C_{\rho}\geqslant 0$  such that every representation in $U_{\rho}$ is ($K_\rho, C_\rho$) $A$-stable. Proposition \ref{displacing} now gives us $J_\rho$ and $B_\rho$ such that for every $\sigma\in U_{\rho}$ and all $a\in A$, 
$$\frac{1}{J_\rho}l_S(a)-B_\rho \leqslant l(\sigma(a)) \leqslant J_\rho l_S(a)+B_\rho,~~~~ \forall \sigma\in U_{\rho},~~ \forall a\in A. $$

Since $R$ is compact, we can consider a finite subcover of $(U_\rho)_{\rho\in R}$ : there are $\rho_1,...,\rho_k \in R$ such that $R\subset \bigcup\limits_{i=1}^{k} U_{\rho_i} $. We take $J_R=\max \{J_{\rho_1},...,J_{\rho_k}\}$ and $B_R=\max \{B_{\rho_1},...,B_{\rho_k}\}$.
\end{proof}

\subsection{Dynamics of $Out(\Gamma)$ on $S_A(\Gamma,G)$}

In order to avoid triviality, we now assume that $\Gamma$ is a hyperbolic group such that $Out(\Gamma)$ is infinite.

\begin{defi}\label{wt}
Let $A \subset \Gamma$, we say that $A$ is \textit{testable} if $A$ is invariant by automorphisms of $\Gamma$ and if, for all sequences $(\phi_n)$ of distinct elements of $Out(\Gamma)$, there is an element $a\in A$ such that $\limsup_n ~l_S(\phi_n(a)) = \infty$.
\end{defi}

We note that this property does not depend on the choice of a set of generators $S$.

\begin{thm}\label{auto-test}
If $A\subset \Gamma$ is testable, then $S_A(\Gamma,G)$ is an open set of discontinuity of $X(\Gamma,G)$ for the action of $Out(\Gamma)$.
\end{thm}

\begin{proof}
By contradiction, we suppose that there are a compact subset $R$ of $S_A(\Gamma,G)$ and a sequence $(\phi_n)_n$ of distinct elements of $Out(\Gamma)$ such that $ \phi_n(R) \cap R \ne \emptyset $ for infinitely many $n\in\mathbb{N}$. Wlog, suppose that $ \phi_n(R) \cap R \ne \emptyset $ for all $n\in \N$. There is a sequence $(\rho_n)_n$ of elements of $R$ such that $\phi_n \cdot \rho_n =\rho_n \circ \phi_n^{-1} \in R$ for all $n\in \N$. We write $\rho'_n:= \rho_n \circ \phi_n^{-1}$ this other sequence of elements of $R$.
Since $A$ is testable, there is $a\in A$ such that $\limsup_n ~l_S(\phi_n^{-1}(a))=\infty$. Let $J_R$ and $B_R$ be the constants obtained by Corollary \ref{displacing compact}, we have :
$$ \frac{1}{J_R}l_S(\phi_n^{-1}(a))-B_R \leqslant l(\rho_n (\phi_n^{-1}(a)))=l(\rho'_n(a)) \leqslant J_Rl_S(a)+B_R \qquad \forall n \in \N$$
We get a contradiction since $\limsup_n~ l_S(\phi_n^{-1}(a))=\infty$.
\end{proof}

\subsection{When $A$ is a characteristic subgroup}

Recall that $\Gamma$ denotes a hyperbolic group, we now make the additional assumption that $\Gamma$ is non-elementary.

We aim at proving Theorem \ref{Awd} whose proof is similar to that of Proposition \ref{wd} which can be found in \cite{Well-displacing}. We start by mentioning one definition and two lemmas from \cite[Section 4]{Well-displacing}.

\begin{defi}
Let $S$ be a finite set of generators for $\Gamma$, $\delta_S$ be the hyperbolicity constant of $C_S(\Gamma)$ and $u,v\in \Gamma$. We say that $(u,v)$ is a ping-pong pair if :
\begin{enumerate}
\item $\min(|u|_S,|v|_S)\geqslant 100\delta_S$
\item $\langle u^{\pm 1},v^{\pm 1} \rangle_e \leqslant \frac{1}{2}\min(|u|_S,|v|_S) -20\delta_S$
\item $\langle u, u^{-1} \rangle_e \leqslant \frac{1}{2}|u|_S -20\delta_S$ and $\langle v, v^{-1} \rangle_e \leqslant \frac{1}{2}|v|_S -20\delta_S$
\end{enumerate}

where $|.|_S=d_S(\mathrm{e},.)$ denotes the word length with respect to $S$, and  $\langle .,. \rangle_e $ denotes the Gromov product based at $\mathrm{e}$.
\end{defi}

\begin{rems}
A ping-pong pair generates a free subgroup of rank two (see \cite{Delzant} for a proof).\\
\end{rems}

Recall that if $\gamma \in \Gamma$ has infinite order, then $\gamma^-$ and $\gamma^+$ denote its distinct fixed points in $\partial C_S(\Gamma)$.

\begin{lem}\label{L1}
If $\gamma_0\in \Gamma$ and $h\in \Gamma$ are two elements of infinite order and $\{h^-,h^+\}\cap \{\gamma_0^-,\gamma_0^+\}=\emptyset$, then there is $N\in \mathbb{N}$ such that $(\gamma_0^N, h \gamma_0^{-N} h^{-1})$ is a ping-pong pair.
\end{lem}

\begin{lem}\label{L2}
If $u,v\in \Gamma$ is a ping-pong pair, then there is some $\alpha\geqslant 0$ such that :
$$|\gamma|_S\leqslant 3 \max \{ l_S(\gamma), l_S(u\gamma), l_S(v\gamma) \} + \alpha~~~~ \forall \gamma\in \Gamma$$
\end{lem}

We deduce next proposition :

\begin{prop}\label{P1}
Let $A$ be a subgroup of $\Gamma$ which contains a free subgroup of rank two.\\
There are $\alpha\geqslant 0$ and $a_1,a_2 \in A$ such that 
$$|\gamma|_S \leqslant 3 \max \{ l_S(\gamma), l_S(a_1\gamma), l_S(a_2\gamma) \} + \alpha ~~~~ \forall \gamma\in \Gamma$$
\end{prop}

\begin{proof}
Let $b_1,b_2\in A$ two elements of $A$ generating a free subgroup of rank 2. We have $\{b_1^-,b_1^+\}\cap \{b_2^-,b_2^+\}=\emptyset$, hence Lemma \ref{L1} gives an $N\in \mathbb{N}$ such that $(b_1^N, b_2 b_1^{-N} b_2^{-1})$ is a ping-pong pair. Take $a_1=b_1^N$, $a_2= b_2 b_1^{-N} b_2^{-1}$ and use Lemma \ref{L2}.
\end{proof}

We will also use the following lemma which is very close to Lemma \ref{l2}.

\begin{lem}\label{l3}
Let $\rho\in R(\Gamma,G)$ be a representation, let $x\in X$ be a basepoint, let $K,K'\geqslant 1$ and $C,C'\geqslant 0$ be some constants and let $l'$ be a $(K',C')$ quasi-geodesic of $C_S(\Gamma)$.\\
There are constants $K''\geqslant 1$ and $C''\geqslant 0$ such that, if $\Or(l')$ is a $(K,C)$ quasi-geodesic of $X$, then for any geodesic $l$ of $C_S(\Gamma)$ with the same endpoints than $l$, $\Or(l)$ is a $(K'',C'')$ quasi-geodesic of $X$.
\end{lem}

\begin{thm}\label{A-bien}
Let $\rho\in R(\Gamma,G)$ be a representation of a non-elementary hyperbolic group in an isometry group of a Gromov-hyperbolic, proper and geodesic metric space.\\
If $A$ is a characteristic subgroup, the representation $\rho$ is $A$-stable if and only if $\rho$ is $A$-well-displacing.
\end{thm}

\begin{proof}
The direct implication is given by Proposition \ref{displacing}.\\
Following the proof of \cite[Lemma 4.0.4]{Well-displacing}, we consider an $A$-well-displacing representation $\rho\in R(\Gamma,G)$ : there are $J\geqslant 1$ and $B \geqslant 0$ such that 
$$\forall a\in A, ~~~~  \frac{1}{J}l_S(a) - B \leqslant l(\rho(a)) \leqslant Jl_S(a) + B$$\\

Specifically, for all $x\in X$ and for all $a\in A$, $d(x, \rho(a)(x)) \geqslant \frac{1}{J}l_S(a) - B$.\\
 
Since a non trivial characteristic subgroup of a non-elementary hyperbolic group always contains a free subgroup of rank 2, Lemma \ref{P1} gives the existence of elements $a_1,a_2 \in A$ and of a constant $\alpha\geqslant 0$ such that : 
$$\forall a\in A, ~~~~ \frac{1}{3}|a|_S -\frac{1}{3}\alpha \leqslant \max \{l_S(a),l_S(a_1a), l_S(a_2a) \}. $$

For all $x\in X$ and for all $a\in A$, we have :

\begin{equation*} 
\begin{split}
d(x, \rho(a)(x)) & \geqslant \max \{ d(x,\rho(a)(x)), d(x,\rho(a_1a)(x)), d(x,\rho(a_2a)(x)) \}- \max \{d(x,\rho(a_1)(x)), d(x,\rho(a_2)(x)) \} \\
 & \geqslant \max \{ l(\rho(a)), l(\rho(a_1a)), l(\rho(a_2a)) \} - \max \{d(x,\rho(a_1)(x)), d(x,\rho(a_2)(x)) \} \\
 & \geqslant \frac{1}{J} \max \{ l_S(a), l_S(a_1a), l_S(a_2a) \} - B - \max \{d(x,\rho(a_1)(x)), d(x,\rho(a_2)(x)) \} \\
 & \geqslant \frac{1}{3J} |a|_S -\frac{\alpha}{3J} - B - \max \{d(x,\rho(a_1)(x)), d(x,\rho(a_2)(x)) \}
\end{split}
\end{equation*}

For all $a\in A_{\infty}$ which have minimal word length in their conjugacy class, $l_a=\{a^n\}$ is a $(K',C')$ quasi-geodesic of $C_S(\Gamma)$ with constants $K'\geqslant 1$ and $C'\geqslant 0$ independent of $a$ according to Lemma \ref{l1}.\\
Using now the previous inequality and recalling that $\Or$ is Lipschitz, we can find uniform constants $K\geqslant 1$ and $C\geqslant0$ such that $\Or(l_a)$ is a $(K,C)$ quasi-geodesic for all $a\in A_\infty$.
Lemma \ref{l3} now implies that there are uniform constants $K''\geqslant 1$ and $C''\geqslant 0$ such that the image by $\Or$ of every geodesic of $L_S(A)$ is a $(K'',C'')$ quasi-geodesic of $X$.
\end{proof}

\subsection{Geometric interpretation when $G=PSL_2(\mathbb{C})$}

This section ends with a geometric interpretation due to Lee in \cite{Lee}, of discrete and faithful $A$-stable representations when $\Gamma$ is torsion-free and $G=PSL_2(\mathbb{C})$. This result will be used in Section 6 and may help some readers to be more familiar with those representations.\\

Let $\Gamma$ be a torsion-free hyperbolic group and $G=PSL_2(\mathbb{C})$ which is identified with the group of orientation-preserving isometries of the 3-dimensional hyperbolic space $\mathbb{H}^3$.

If $\rho\in R(\Gamma,G)$ and $a\in A$ are such that $l(\rho(a))>0$, then we denote by $L_\rho(a)$ the geodesic connecting $\rho(a)^-$ to $\rho(a)^+$ in $\mathbb{H}^3$.

\begin{defi}
We say that $\rho\in R(\Gamma,G)$ is \textit{$A$-convex cocompact} if 
\begin{enumerate}
\item $\rho(a)$ is hyperbolic for all $a\in A\setminus \{\mathrm{e}\}$,
\item There is a $\rho(\Gamma)$-invariant subset $\Omega\subset \mathbb{H}^3$ such that $\Omega \diagup \rho(\Gamma)$ is compact and $L_\rho(A):=\bigcup_{a\in A\setminus \{\mathrm{e}\}} L_\rho(a)$ is contained in $\Omega$.
\end{enumerate}
\end{defi}

The following equivalence is due to Lee in \cite[Lemma III.10]{Lee} :

\begin{prop}\label{A-cc}
Let $\rho\in R(\Gamma,G)$ be a discrete and faithful representation.
The representation $\rho$ is $A$-stable if and only if $\rho$ is $A$-convex cocompact.
\end{prop}

\begin{rem}
It is not hard to show that a representation is $\Gamma$-convex cocompact if and only if it is convex cocompact.\\
As Lee's proof uses the tameness property of hyperbolic 3-manifolds, this need not extend to higher dimensional settings. 
\end{rem}

\section{Actions on $\mathbb{R}$-trees and test subsets}

In this section, $\Gamma$ will denote a hyperbolic group such that its outer automorphism group $Out(\Gamma)$ is infinite and $S$ will denote a finite set of generators for $\Gamma$.\\

Before giving examples of $A$-stable representations, we focus on certain subsets of $\Gamma$ and ask for their testability (see Definition \ref{wt}). In order to do so, we recall some basic facts about $\R$-trees and isometric actions on $\R$ trees and use a theorem of Paulin in \cite{Paulin2} that asserts that if $Out(\Gamma)$ is infinite, then $\Gamma$ acts isometrically on an $\R$-tree without global fixed points. We take advantage of this new approach to prove Theorem \ref{Can}.

\subsection{$\mathbb{R}$-trees}

We refer the reader to \cite{CullerMorgan} or \cite{Paulin} for a more detailed introduction on $\mathbb{R}$-trees.

\begin{defi}
Let $(X,d)$ be a metric space and let $x,y\in X$. An \textit{arc} from $x$ to $y$ is the image of a topological embedding $\alpha: [a,b] \rightarrow X$ of a closed interval $[a,b]$ such that $\alpha(a)=x$ and $\alpha(y)=b$.
\end{defi}

\begin{defi}
A \textit{$\mathbb{R}$-tree} is a non-empty metric space $(T,d)$ such that for every $x,y\in T$, there is a unique arc connecting $x$ to $y$ and this arc is a geodesic segment.\\
We say that $T' \subset T$ is a \textit{subtree} if $T'\ne \emptyset$ and $T'$ is connected.
\end{defi}

It is a straightforward exercise to show that $\mathbb{R}$-trees are $0$-hyperbolic metric spaces.
\\

Let $G=Isom(T)$ be the isometry group of an $\mathbb{R}$-tree $T$. We denote by $l_T(g)=l(g)$ the translation length of $g\in G$.\\

The \textit{characteristic set of $g$} is defined as the set

$$C_g = \{x\in T ~|~ d(x,g(x))=l_T(g)\}$$

If $l_T(g)=0$, then $g$ is elliptic and $C_g$ is the set of fixed points of $g$.\\
If $l_T(g)>0$, then $g$ is hyperbolic, $C_g$ is isometric to $\mathbb{R}$ and is called the \textit{axis} of $g$. Furthermore, $g$ acts on $C_g$ by translation of length $l_T(g)$.\\

Proofs of the previous and following statements can be found in \cite[Section 1]{CullerMorgan}.

\begin{prop}\label{cha set}
For all $g,h\in G$, the following holds :

\begin{itemize}
\item $C_g$ is a closed subtree of $T$, invariant by the action of $g$.
\item $C_{g^{-1}}=C_g$
\item $C_{ghg^{-1}}= g(C_h)$
\end{itemize}
\end{prop}

We now give two propositions which will be useful and frequently used in what follows (see \cite[ Proposition 1.2]{CullerMorgan} and \cite[Proposition 1.8]{Paulin} respectively).

\begin{prop}\label{tree}
If $T_1,T_2,T_3$ are three closed subtrees of an $\mathbb{R}$-tree $T$ such that $T_i\cap T_j$ is not empty for all $1 \leqslant i,j \leqslant 3$, then $T_1\cap T_2 \cap T_3$ is not empty.
\end{prop}

\begin{rem}
By induction, if $T_1,...,T_k$ are $k$ closed subtrees of an $\mathbb{R}$-tree $T$ such that $T_i\cap T_j$ is not empty for all $1 \leqslant i,j \leqslant k$, then $T_1\cap ... \cap T_k$ is not empty.
\end{rem}

\begin{prop}\label{elliptique}
If $g,h$ are two elliptic isometries of an $\mathbb{R}$-tree $(T,d)$, then $l_T(gh)=2d(C_g,C_h)$.\\
In particular, if $g$, $h$ and $gh$ are elliptic, then $C_g\cap C_h\ne \emptyset$.
\end{prop}

\subsection{Isometric actions on $\mathbb{R}$-trees}

A \textit{$\Gamma$-space} is a pair $(X, \rho)$ where $X$ is a metric space and $\rho\in Hom(\Gamma,Isom(X))$ is an isometric action of $\Gamma$ on $X$.\\

A \textit{$\Gamma$-tree} is a $\Gamma$-space $(T,\rho)$ where $T$ is an $\mathbb{R}$-tree. If there is no ambiguity, we denote the translation length of $\rho(\gamma)$ by $l_T(\gamma)$.\\

A $\Gamma$-tree $(T,\rho)$ is \textit{trivial} if there is a global fixed point in $T$.\\

Let $\B_S = S \cup \{ss'~|~ s,s'\in S$, $s\ne s' \}$ be the set of generators and product of distinct generators.

\begin{prop}\label{trivial}
Let $\Gamma$ be a finitely generated group and $(T,\rho)$ be a $\Gamma$-tree. The following are equivalent :
\begin{itemize}
\item $(T,\rho)$ is trivial,
\item $l_T(\gamma)=0 ~~ \forall \gamma\in \Gamma$
\item $l_T(\gamma)=0 ~~ \forall \gamma\in \B_S$
\end{itemize}
\end{prop}

\begin{proof}
If $(T,\rho)$ is trivial, every element of $\Gamma$ fixes a point, hence is elliptic and $l_T(\gamma)=0$ for all $\gamma\in \Gamma$.\\

Suppose now that $l_T(\gamma)=0 ~~ \forall \gamma\in \B_S$.\\
Let $s,s'\in S$ be such that $s\ne s'$. Since $s$, $s'$ and $ss'$ are elliptic, the set of fixed points $C_{s}$ and $C_{s'}$ have non empty intersection according to Proposition \ref{elliptique}. Proposition \ref{tree} now implies that the intersection $\bigcap_{s\in S} C_{s}$ is non empty : therefore there is a point of $T$ fixed by all generators of $\rho(\Gamma)$, hence by all elements of $\rho(\Gamma)$.
\end{proof}

We say that an arbitrary group is \textit{small} if it does not contain a free subgroup of rank 2.\\
We say that a $\Gamma$-tree $(T,\rho)$ is \textit{small} if the pointwise stabilizer of any non-degenerate segment is small.\\

\subsection{Sequence of automorphisms and convergence of based $\Gamma$-spaces}

We say that a symmetric application $d : \Gamma \times \Gamma \rightarrow [0,\infty[$, vanishing on the diagonal and satisfying the triangular inequality is a \textit{pseudometric} on $\Gamma$.
Let us consider  the space $\mathcal{D}(\Gamma)$ of all pseudometrics on $\Gamma$ endowed with the compact-open topology.\\
$_Gamma$ acts diagonally on $\Gamma \times \Gamma$ by left multiplication. We denote by $\mathcal{ED}(\Gamma) \subset \mathcal{D}(\Gamma)$ the subset of pseudometrics on $\Gamma$ which are invariant under this action.\\
$\mathbb{R}^+$ acts freely on $\mathcal{ED}(\Gamma)$ by scaling, we denote by $\mathcal{PED}(\Gamma)$ the quotient space of projectivized $\Gamma$-equivariant pseudometrics endowed with quotient topology.

A \textit{based $\Gamma$-space} is a triple $(X,x,\rho)$ where $(X,\rho)$ is a $\Gamma$-space, $\rho\in Hom(\Gamma,Isom(X))$ is an isometric action of $\Gamma$ on $X$ and $x\in X$ is a base point which is not fixed by all elements of $\rho(\Gamma)$.\\

A based $\Gamma$-space $(X,x,\rho)$ induces an equivariant pseudo-metric $d_{(X,x,\rho)}$ on $\Gamma$ by setting :
$$ d_{(X,x,\rho)}(\gamma,\gamma') := d_X(\rho(\gamma)(x), \rho(\gamma')(x)) ~~~~ \forall \gamma, \gamma' \in \Gamma$$

Let $((X_n,x_n,\rho_n))_n$ be a sequence of based $\Gamma$-spaces. We say that the sequence $((X_n,x_n,\rho_n))_n$ \textit{converges} to the based $\Gamma$-space $(X,x,\rho)$ if $[d_{(X_n,x_n,\rho_n)}]$ converges to $ [d_{(X,x,\rho)}]$ in $\mathcal{PED}(\Gamma)$.\\

From now on, $(\varphi_n)$ will denote a sequence of elements of $Aut(\Gamma)$ that projects onto a sequence of distinct elements of $Out(\Gamma)$.\\

For all $n\in \mathbb{N}$, we introduce the application $f_n: C_S(\Gamma) \rightarrow \mathbb{N}$ such that $f_n(x)= \max_{s\in S} d_S(x, \varphi_n(s) x)$ for all $x\in C_S(\Gamma)$. \\
Let $\lambda_n= \min \{f_n(x) ~|~ x\in C_S(\Gamma)\}$ and choose arbitrarily $x_n\in C_S(\Gamma)$ that realize this minimum.
Moreover, we will denote by $\rho_0\in Hom(\Gamma,Isom(C_S(\Gamma)))$ the natural action of $\Gamma$ on $C_S(\Gamma)$, and we write  $\rho_n:=\rho_0\circ \varphi_n \in Hom(\Gamma,Isom(C_S(\Gamma)))$.\\

The sequences $(x_n)\in C_S(\Gamma)^{\mathbb{N}}$, $(\lambda_n)\in \mathbb{N}^{\mathbb{N}}$ and $(\rho_n)\in Hom(\Gamma,Isom(C_S(\Gamma)))^{\mathbb{N}}$ will be called \textit{sequences induced by $(\varphi_n)$}.

\begin{prop}\label{infini}
The sequence of integers $(\lambda_n)$ induced by $(\varphi_n)$ is not bounded.
\end{prop}

\begin{proof}
Ad absurdum, suppose that $(\lambda_n)$ is bounded by $M$.
Then for any vertex $y_n$ closest to $x_n$, we would have $d_S(e,y_n^{-1}\varphi_n(s)y_n)=d_S(y_n,\varphi_n(s)y_n)\leqslant M+2$ for all $s\in S$ and $n\in \mathbb{N}$.\\
Since there are only finitely many vertices in the ball of radius $M+2$ about $e$, this bound would imply the existence of integers $n,m\in \mathbb{N}$, $n\ne m$ such that $y_n^{-1}\varphi_n(s)y_n =y_m^{-1}\varphi_m(s)y_m$ for all $s\in S$. Thereby, $\varphi_n$ and $\varphi_m$ would be equal in $Out(\Gamma)$, this is a contradiction. 
\end{proof}

Next theorem can be found in \cite[Theorem 3.9]{Best} :

\begin{thm}\label{Best}
Let $(x_n)$, $(\rho_n)$ and $(\lambda_n)$ be the sequences induced by $(\varphi_n)$.\\
The sequence $((C_S(\Gamma),x_n,\rho_n))$ of based $\Gamma$-spaces converges to $(T,x,\rho)$, where $(T,\rho)$ is a small and non-trivial $\Gamma$-tree and $x\in T$.\\

More precisely, $$d_{(T,x,\rho)}(.,.)= \lim_{n \rightarrow \infty} \frac{d_{(C_S(\Gamma),x_n,\rho_n)}}{\lambda_n}(.,.) $$
Moreover, $$l_T(\rho(\gamma))=\lim_{n \rightarrow \infty} \frac{l_S(\varphi_n(\gamma))}{\lambda_n} \qquad \forall \gamma\in \Gamma$$
\end{thm}

The based $\Gamma$-tree $(T,x,\rho)$ given by this theorem will be called \textit{based $\Gamma$-tree induced by $(\varphi_n)$}.

\begin{cor}\label{bornée}
Let $(T,x,\rho)$ be the based $\Gamma$-tree induced by $(\varphi_n)$.
If $\gamma\in \Gamma$ is such that $(l_S(\varphi_n(\gamma)))$ is bounded, then $\rho(\gamma)$ is elliptic.
\end{cor}

\begin{proof}
Recall that the sequence $(\lambda_n)$ of integers induced by $(\varphi_n)$ is not bounded according to Proposition \ref{infini}. The last equality of Theorem \ref{Best} yields that $l_T(\rho(\gamma))=0$, hence $\rho(\gamma)$ is an elliptic isometry of $T$.
\end{proof}

\subsection{Testable sets and test subsets}

\begin{defi}
We say that $B\subset \Gamma$ is a \textit{test subset} if, for all sequences $(\phi_n)_n$ of distinct elements of $Out(\Gamma)$,  there is $b\in B$ such that $\limsup_n~ l_S(\phi_n(b)) = \infty$. 
\end{defi}

This property does not depend on the choice of the finite set of generators $S$. Note that if $A\subset \Gamma$ is invariant by automorphisms and $A$ contains a test subset, then $A$ is testable.

\subsubsection{$\B_S$ is a test subset}

Recall that $\B_S= S \cup \{ss'~|~ s,s'\in S$, $s\ne s' \}$ denotes the set of generators and product of distinct generators. Richard Canary asked the following question in the appendix of \cite{Canary} : \\

Is it true that for all hyperbolic group $\Gamma$ and all $M\geqslant 0$, the set 
$$\{\phi\in Out(\Gamma)  ~|~ l_S(\phi(b))\leqslant M ~~ \forall b\in \B_S\}$$
is finite ?\\

The next proposition gives an affirmative answer to this question, and can be used to prove the proper discontinuity of the action of $Out(\Gamma)$ on the set of some $A$-stable representations (see next section).

\begin{thm}\label{test}
For all hyperbolic groups $\Gamma$ such that $Out(\Gamma)$ is infinite and all finite generating sets $S$, $\B_S= S \cup \{ss'~|~ s,s'\in S$, $s\ne s' \}$ is a test subset.
\end{thm}

\begin{proof}
Ad absurdum, suppose that there is a sequence $(\phi_n)$ of distinct elements of $Out(\Gamma)$ such that for all $b\in \B_S$, $(l_S(\phi_n(b)))$ is bounded. Denote by $(\varphi_n)$ any sequence of elements of $Aut(\Gamma)$ which projects onto $(\phi_n)$ and let $(T,x,\rho)$ be the based $\Gamma$-tree induced by $(\varphi_n)$.\\
Hence $\rho(b)$ would be elliptic for all $b\in \B_S$ according to Corollary \ref{bornée}. Proposition \ref{trivial} now implies that $(T,\rho)$ is trivial, contradicting Theorem \ref{Best}.
\end{proof}

\subsubsection{The commutator set is testable}

The subset $\mathcal{C}=\{[g,h]~|~ g,h\in \Gamma \}$ of commutators of $\Gamma$ is invariant by automorphisms. In this section, we show that it is testable.\\

A $\Gamma$-tree $(T,\rho)$ is \textit{minimal} if there is no $\rho(\Gamma)$-invariant subtree $T'$ properly contained in $T$. If $(T,\rho)$ is non-trivial, there is a unique minimal invariant subtree $T'$ : the union of all axes of hyperbolic elements of $\rho(\Gamma)$ (see \cite[Proposition 2.4]{Paulin} for a proof).\\

We say that a $\Gamma$-tree $(T,\rho)$ is \textit{irreducible} if $(T,\rho)$ is non trivial and $\rho(\Gamma)$ does not fix any point of $\partial T$.\\

Culler and Morgan proved the following (\cite[Lemma 5.1]{CullerMorgan} and \cite[Corollary 2.3]{CullerMorgan}) :

\begin{lem}\label{p1}
If $\Gamma$ is a non-elementary hyperbolic group and $(T,\rho)$ is a non-trivial, minimal and small $\Gamma$-tree, then $(T,\rho)$ is irreducible.
\end{lem}

\begin{lem}\label{p2}
If $(T,\rho)$ is an irreducible $\Gamma$-tree, then there are $g,h\in \Gamma$ such that $l_T([g,h]) > 0$.  
\end{lem}

We can deduce

\begin{prop}\label{well-test}
If $\Gamma$ is a non-elementary hyperbolic group, then the set $\mathcal{C}$ of commutators of $\Gamma$ is testable.
\end{prop}

\begin{proof}
Ad absurdum, suppose there is a sequence $(\phi_n)$ of distinct elements of $Out(\Gamma)$ and such that $(l_S(\phi_n(w)))$ is bounded for all $w\in \mathcal{C}$. Denote by $(\varphi_n)$ any sequence of elements of $Aut(\Gamma)$ which projects onto $(\phi_n)$ and let $(T,x,\rho)$ be the based $\Gamma$-tree induced by $(\varphi_n)$.\\
$(T,\rho)$ is a non-trivial and small $\Gamma$-tree according to Theorem \ref{Best}. Up to restricting to $T'$, where $T'$ denotes the union of all axes of hyperbolic elements of $\rho(\Gamma)$, we can suppose that the action is minimal. Hence $(T',\rho)$ is irreducible according to Lemma \ref{p1} and Lemma \ref{p2} now states that there is $w\in \mathcal{C}$ such that $\rho(w)$ is hyperbolic. Corollary \ref{bornée} leads us to a contradiction.
\end{proof}

\section{Examples of $A$-stable representations}

Let $\Gamma$ be a hyperbolic group, let $S$ be a finite set of generators of $\Gamma$ and let $G=Isom(X)$ be the isometry group of a Gromov-hyperbolic, proper and geodesic metric space $X$.\\

In this section, we show that quasi-convex representations and primitive stable representations are special cases of $A$-stable representations and finally give a proof of Proposition \ref{QC}.
We also introduce a new invariant subset of representations, namely the set of commutator-stable representations. Specializing to $G=PSL_2(\mathbb{C})\cong Isom^+(\mathbb{H}^3)$, we prove Theorem \ref{df2}.

If $A\subset \Gamma$ is an $Aut(\Gamma)$-invariant subset, we will denote by $D_A(\Gamma,G)$ the set of conjugacy classes of $A$-well-displacing representations.

\subsection{Quasi-convex representations}

We start by a trivial case, when $A=\Gamma$ :
$$QC(\Gamma,G)=S_{\Gamma}(\Gamma,G)=D_{\Gamma}(\Gamma,G) \subset AH(\Gamma,G)$$

where $QC(\Gamma,G)$ is the set of conjugacy classes of quasi-convex representations and the last equality holds according to Proposition \ref{wd}.\\

The group $\Gamma$ is testable since it contains the subset $\B_S= S \cup \{ss'~|~ s,s'\in S$, $s\ne s' \}$ which is a test subset according to Theorem \ref{test}. According now to Theorem \ref{auto-test}, $QC(\Gamma,G)$ is an open set of discontinuity of $X(\Gamma,G)$, this proves Proposition \ref{QC} as promised.

\subsection{Primitive stable representations of a free group}

Let $\Gamma=F_k$ be a free group of rank $k\geqslant 2$, $S$ be a minimal set of generators and $\p=\{\varphi(s)~|~ \varphi\in Aut(F_k), s\in S \}$ be the set of \textit{primitive} elements of $F_k$.
The subset $\p$ is clearly invariant by automorphisms of $F_k$ and contains the test subset $\B_S$, therefore $\p$ is testable.\\

The set $S_{\p}(F_k,G)$ of \textit{primitive-stable} representations is an open set of discontinuity of $X(F_k,G)$ for the action of $Out(F_k)$ according to Theorem \ref{auto-test}.\\

Minsky (\cite{Minsky}, Theorem 1.1) actually proved for $G=PSL_2(\mathbb{C})$ :

\begin{thm}
$S_{\p}(F_k,PSL_2(\mathbb{C}))$ is an open set of discontinuity of $X(F_k,PSL_2(\mathbb{C}))$ for the action of $Out(F_k)$ which is stricly larger than the set $CC(F_k,PSL_2(\mathbb{C}))$ of convex cocompact representations.
\end{thm}

We give a recap of the previous inclusions between $Out(\Gamma)$-invariant subsets of $X(F_k,PSL_2(\mathbb{C}))$ :

$$CC(F_k,PSL_2(\mathbb{C})) \varsubsetneq AH(F_k,PSL_2(\mathbb{C})) \varsubsetneq S_{\p}(F_k,PSL_2(\mathbb{C})) \subset D_{\p}(F_k,PSL_2(\mathbb{C})) $$

We know that primitive stable representations are $\p$-well-displacing according to Proposition \ref{displacing} but we don't know if the converse holds in general. However, the answer is known when $k=2$ (see \cite{Series} or \cite{Xu} when $G=PSL_2(\mathbb{C})$ and \cite[Theorem 1.3]{Schlich} for the following general result) :

\begin{prop}
Let $F_2$ be the free group of rank 2 and $G$ the isometry group of a Gromov-hyperbolic, proper and geodesic metric space.\\
A representation $\rho\in R(F_2,G)$ is $\p$-stable if and only if $\rho$ is $\p$-well-displacing, i.e., 
$$S_{\p}(F_2,G)=D_{\p}(F_2,G)$$
\end{prop}

\subsection{Commutator-stable representations}

Let $\Gamma$ be a non-elementary hyperbolic group and let $\mathcal{C}=\{[g,h]~|~ g,h\in \Gamma \}$ denotes the subset of commutators of $\Gamma$.

\begin{prop}\label{CS1}
The set $S_\mathcal{C}(\Gamma,G)$ of conjugacy classes of commutator-stable representations is an open set of discontinuity of $X(\Gamma,G)$ for the action of $Out(\Gamma)$.
\end{prop}

\begin{proof}
The proof is a direct combination of Proposition \ref{well-test} and Theorem \ref{auto-test}.
\end{proof}

We now recall some basic facts on convergence actions as defined by Gehring and Martin and use it to get informations on translation lengths. We recall that an isometry $g\in G$ extends to a homeomorphism of the Gromov compactification $\overline{X}$, which will again be denoted $g$.\\

\begin{defi}
Let $(g_n)$ be a sequence of elements of $G$ and $a,b\in \partial X$.
We say that $(g_n)$ is a \textit{convergence sequence with base $(a,b)$} if the restriction of $(g_n)$ to $\overline{X}\setminus \{a\}$ converges to $b$ uniformly on compact subsets.
\end{defi}

One can check that : 

\begin{prop}\label{collapse}
Let $a,b\in \partial X$ and $(g_n)$ be a convergence sequence with base $(a,b)$ :
\begin{itemize}
\item The sequence $(g_n^{-1})$ is a convergence sequence with base $(b,a)$.
\item If $h : X \rightarrow X$ is a homeomorphism, then $(h\circ g_n \circ h^{-1})$ is a convergence sequence with base $(h(a), h(b))$.
\item If $c,d\in \partial X$, $b\ne c$ and $(h_n)$ is a convergence sequence with base $(c,d)$, then $(h_ng_n)$ is a convergence sequence with base $(a,d)$.
\end{itemize}
\end{prop}

A proof of the next theorem can be found in \cite[Theorem 3A]{Tukia} :

\begin{thm}\label{Tukia}
Let $(g_n)$ be a sequence of elements of $G$. One of the two following proposals is verified :
 
\begin{itemize}
\item The sequence $(g_n)$ is relatively compact.
\item Up to subsequence, there are $a,b\in \partial X$ such that $(g_n)$ is a convergence sequence with base $(a,b)$.
\end{itemize}

\end{thm}

We will need the following lemmas (see \cite[Corollary 2E]{Tukia} for a proof of the first one) :

\begin{lem}\label{ll1}
Let $a,b\in \partial X$ such that $a\ne b$. If $(g_n)$ is a convergence sequence with base $(a,b)$, then $g_n$ is hyperbolic for $n$ large enough and
$$g_n^- \rightarrow a \qquad \text{and} \qquad g_n^+ \rightarrow b$$
where $g_n^-$ and $g_n^+$ denote the distinct fixed points in $\partial X$ of $g_n$. 
\end{lem}

Until the end of this section, we will deliberately confuse the elements of $\Gamma$ with their images in $Isom(C_S(\Gamma))$ so as not to complicate the notations.

\begin{lem}\label{ll2}
Let $a,b\in \partial C_S(\Gamma)$ such that $a\ne b$ and let $(\gamma_n)$ be a sequence of elements of $\Gamma$ which is a convergence sequence with base $(a,b)$. The sequence $(l_S(\gamma_n))_n$ is not bounded.
\end{lem}

\begin{proof}
Let $(\gamma_n)$ be a convergence sequence with base $(a,b)$ such that $a\ne b$ and $l$ be a geodesic of $C_S(\Gamma)$ connecting $a$ to $b$. According to Lemma \ref{ll1}, we can assume up to subsequence, that $\gamma_n$ is hyperbolic for all $n\in \mathbb{N}$. Lemma \ref{l1} gives us uniform constants $K'\geqslant 1, C'\geqslant 0$ such that $\gamma_n$ has a $(K',C')$ quasi-axis $l_n$ connecting $\gamma_n^-$ to $\gamma_n^+$.\\
By contradiction, let us suppose that $(l_S(\gamma_n))_n$ is bounded. Let $x\in l$ and for all $n\in\mathbb{N}$, choose $x_n\in l_n$ such that $d_S(x,x_n)=d_S(x,l_n)$. It follows that 

$$d_S(x,\gamma_n(x)) \leqslant 2d_S(x,x_n) + d_S(x_n,\gamma_n(x_n)) \qquad \forall n\in \mathbb{N}$$

Combining the second assertion of Lemma \ref{ll1} and the Morse lemma, we get that $d_S(x,x_n))$ is a bounded sequence. Hence, the assumption that $d_S(x_n,\gamma_n(x_n))= l_S(\gamma_n)$ is also bounded would contradict the fact that $(\gamma_n)$ is a convergence sequence.
\end{proof}

\begin{lem}\label{sort}
Let $(\gamma_n)$ be a sequence of distinct elements of $\Gamma$.\\
There is some element $h\in \Gamma$ such that $(l_S([\gamma_n,h]))$ is not bounded. 
\end{lem}

\begin{proof}
Since $\Gamma$ acts properly discontinuously on $C_S(\Gamma)$, $(\gamma_n)$ is not relatively compact in $Isom(C_S(\Gamma))$, hence according to Theorem \ref{Tukia}, there are $\gamma^-,\gamma^+ \in \partial C_S(\Gamma)$ such that up to subsequence, $(\gamma_n)$ is a convergence sequence with base $( \gamma^-,\gamma^+)$. Let $h\in \Gamma$ with infinite order such that $\{h^-,h^+\}$ is disjoint from $\{\gamma^-,\gamma^+\}$. According to Proposition \ref{collapse}, the sequence $(h\gamma_n^{-1}h^{-1})$ is a convergence sequence with base $(h(\gamma^+),h(\gamma^-))$, hence $[\gamma_n,h]$ is a convergence sequence with base $(h(\gamma^+), \gamma^+)$ where $\gamma^+ \ne h(\gamma^+)$. We conclude using Lemma \ref{ll2}.
\end{proof}

We are now able to prove the following sufficient condition for a representation to be discrete and have finite kernel :

\begin{prop}\label{CS2}
A $\mathcal{C}$-well-displacing representation of a hyperbolic group $\Gamma$ in any topological isometry group $G$ always has discrete image and finite kernel. 
\end{prop}

\begin{proof}
We actually show the contraposition. Suppose that a representation $\rho\in R(\Gamma,G)$ has not a finite kernel. According to Proposition \ref{inf}, this implies that there is an infinite order element $\gamma\in \Gamma$ such that $\rho(\gamma)=\mathrm{Id}$. If $h\in \Gamma$ is such that $\{h^-,h^+\}\cap \{\gamma^-,\gamma^+\}=\emptyset$, then $[\gamma,h]$ has infinite order whereas $\rho([\gamma,h])=\mathrm{Id}$. It follows that $l_S([\gamma,h]^n) \rightarrow \infty$ and $l(\rho([\gamma,h]^n))=0$ for all $n\in \mathbb{N}$.\\

Suppose now that a representation $\rho\in R(\Gamma,G)$ is not discrete, i.e, there is $(\gamma_n)$ a sequence of elements of $\Gamma$ such that $\rho(\gamma_n)$ is a sequence of distinct elements of $G$ and $\rho(\gamma_n)\rightarrow \mathrm{Id}$. According to Lemma \ref{sort}, there is $h\in \Gamma$ such that $(l_S([\gamma_n,h]))$ is not bounded whereas $l(\rho([\gamma_n,h]))\rightarrow 0$.\\

In both cases, $\rho$ is not $\mathcal{C}$-well-displacing.
\end{proof}

We proved the following inclusions of $Out(\Gamma)$-invariant subsets of $X(\Gamma,G)$ :

$$QC(\Gamma,G) \subset S_\mathcal{C}(\Gamma,G) \subset D_\mathcal{C}(\Gamma,G) \subset AH(\Gamma,G)$$

We now give some new characterizations of convex cocompact representations in $PSL_2(\mathbb{C})$ thanks to Theorem \ref{Sullivan} and Proposition \ref{CS2}.

\begin{thm}\label{major}
Let $\Gamma$ be a non-elementary hyperbolic group, let $\mathcal{C}$ be the subset of commutators of $\Gamma$ and let $\mathcal{D}=[\Gamma,\Gamma]$ be the derived subgroup of $\Gamma$.\\
Let $\rho\in R(\Gamma,PSL_2(\mathbb{C}))$, the following are equivalent :
\begin{itemize}
\item $\rho$ is convex cocompact,
\item $\rho$ is $\mathcal{C}$-stable,
\item $\rho$ is $\mathcal{D}$-well-displacing.
\end{itemize}
\end{thm}

\begin{proof}
The first equivalence comes from Theorem \ref{Sullivan} and the fact that $S_\mathcal{C}(\Gamma,G)$ is an open set which naturally contains $CC(\Gamma,G)$ and is contained in $AH(\Gamma,G)$ according to Proposition \ref{CS2}.\\
For the equivalence between convex cocompact and $\mathcal{D}$-well-displacing representations, recall that a convex cocompact representation is $\Gamma$-well-displacing according to Proposition \ref{wd}, hence $\mathcal{D}$-well-displacing and conversely, a $\mathcal{D}$-well-displacing representation is $\mathcal{D}$-stable according to Theorem \ref{A-bien}, hence $\mathcal{C}$-stable, hence convex cocompact.
\end{proof}

Recalling the geometric interpretation of $A$-stable representations due to Lee (see Proposition \ref{A-cc}), we can deduce the following corollary :

\begin{cor}
Let $\Gamma$ be a torsion-free non-elementary hyperbolic group and let $\mathcal{C}$ be the subset of commutators of $\Gamma$.\\
A representation in $R(\Gamma,PSL_2(\mathbb{C}))$ is convex cocompact if and only if it is $\mathcal{C}$-convex cocompact.
\end{cor}

\section{Appendix}

Let $F_k$ denotes the free group of rank $k\geqslant 3$, we already know that its derived subgroup $[F_k,F_k]$ is testable (it is a consequence of Proposition \ref{well-test}).
In this appendix, we show that the derived subgroup $[F_k,F_k]$ actually contains a finite test subset.

Let $S=\{x_1,...,x_k\}$ be a finite and minimal set of generators and $\B_S=\{x_i,x_ix_j~|x_i,x_j\in S$, $i\ne j \}$ be the subset of elements of $S$ and product of distinct elements of $S$. \\
We define the following finite subsets of the derived subgroup $[F_k,F_k]$ :

\begin{itemize}
\item $\w_{1,S}=\w_1 = \{[b,b'] ~|~  b,b' \in \B_S \}$
\item $\w_{2,S}=\w_2 =\{ [b,w] ~|~  b\in \B_S, w\in \w_1 \}$
\item $\w_{3,S}=\w_3 = \{ ww' ~|~  w,w' \in \w_1 \}$
\item $\mathcal{W}_S =\w_1 \cup \w_2 \cup \w_3$
\end{itemize}

\begin{prop}\label{W}
$\mathcal{W}_S$ is a test subset.
\end{prop}

In order to prove this proposition, we will need the following lemma.

\begin{lem}\label{he}
Let $T$ be an $\R$-tree and $g,h\in Isom(T)$ such that $g$ is hyperbolic and $h$ is elliptic.\\ 
If $l_T([g,h])=0$, then $C_g\cap C_h$ is a geodesic segment of length at least $l_T(g)$.
\end{lem}

\begin{proof}
We write $A_g=C_g$ the axis of $g$ on which $g$ acts by translation and we recall that $C_h$ is the set of fixed points of $h$.\\
The commutator $[g,h]=ghg^{-1}h^{-1}$ is the product of two elliptic elements, namely $ghg^{-1}$ and $h^{-1}$ with set of fixed points $g(C_h)$ and $C_h$ respectively according to Proposition \ref{cha set}.\\

The proof is by contraposition :\\

If $A_g\cap C_h=\emptyset$, then $g(C_h)$ and $C_h$ are disjoint (see \cite[1.7]{CullerMorgan}).\\
If $A_g\cap C_h\ne \emptyset$ and $A_g\cap C_h$ is a geodesic segment of length strictly less than $l_T(g)$, then the fact that $g$ acts by translation of length $l_T(g)$ on $A_g$ yields that $A_g\cap g(C_h)\ne \emptyset$ and $A_g\cap C_h \cap g(C_h)= \emptyset$. Proposition \ref{tree} now implies that $C_h \cap g(C_h)= \emptyset$.\\

In both cases, Proposition \ref{elliptique} implies that $l_T(ghg^{-1}h^{-1})>0$.
\end{proof}

\begin{proof}[Proof of Proposition \ref{W}]
Let $(\varphi_n)$ be a sequence of $Aut(F_k)$ that projects on a sequence of distinct elements of $Out(F_k)$. Let $(T,x,\rho)$ be the based $F_k$-tree induced by $(\varphi_n)$ in Theorem \ref{Best}.\\

By contradiction, suppose that $(l_T(\varphi_n(w)))$ is bounded for all $w\in \mathcal{W}_S$. Corollary \ref{bornée} implies that $l_T(w)=l(\rho(w))=0$ for all $w\in \w_S$.\\
Since $(T,\rho)$ is not a trivial $F_k$-tree, there is $b\in \B_S$ such that $l_T(b)>0$ according to Proposition \ref{trivial}. We denote by $A_b=C_b$ the axis of $b$ for its action on $T$, we recall that $b$ acts on $A_b$ by translation of length $l_T(b)$.\\
Let $w_1,w_2,w_3\in \w_1$ be 3 distinct elements such that the subgroup $<w_i,w_j>$ generated by $w_i$ and $w_j$ is isomorphic to a free group of rank 2 for all $i\ne j$. We denote by $C_i=C_{w_i}$ the set of fixed points of $w_i$ for its action on $T$.\\

Since $l_T(w_i)=l_T(w_j)=l_T(w_iw_j)=0$ for all $i\ne j$ by hypothesis, Proposition \ref{elliptique} yields $C_i\cap C_j\ne \emptyset$ for all $i\ne j$.\\
Moreover, $l_T([b,w_i])=0$, hence Lemma \ref{he} implies that $A_b\cap C_i$ is a geodesic segement of $T$ of length at least $l_T(b)$ for all $i$.\\
Proposition \ref{tree} now implies that $A_b\cap C_i \cap C_j \ne \emptyset$ for all $i\ne j$.\\
If $A_b\cap C_1\cap C_2$  is a geodesic segment of positive length, we get a contradiction since $(T,\rho)$ is a small $F_k$-tree and the subgroup $<w_1,w_2>$ is free of rank $2$.\\
If $A_b\cap C_1\cap C_2$ is just a point, then $A_b\cap C_3\cap C_2$ or $A_b\cap C_3 \cap C_1$ must be a geodesic segment of positive length, contradicting once more the hypothesis that $(T,\rho)$ is a small $F_k$-tree.\\

Conclusion : for all sequence $(\phi_n)$ of  distinct elements of $Out(F_k)$, there is $w\in \w_S$ such that $\limsup_n l_S(\phi_n(w))=\infty$.
\end{proof}

\end{document}